\documentclass[10pt]{amsart}
\usepackage{amsmath, amssymb}
\usepackage{mathrsfs}
\usepackage{paralist}
\usepackage{graphics} 
\usepackage{epsfig} 
\usepackage{graphicx}  
\usepackage{epstopdf}
 \usepackage[colorlinks=true]{hyperref}
   
\hypersetup{urlcolor=blue, citecolor=red}

\newtheorem{theorem}{Theorem}[section]
\newtheorem{corollary}{Corollary}

\newtheorem{lemma}[theorem]{Lemma}
\newtheorem{proposition}{Proposition}

\theoremstyle{definition}

\title[Neumann heat kernels] {Observations on Gaussian upper bounds\\  for Neumann heat kernels}

\author[Mourad Choulli, Laurent Kayser and El Maati Ouhabaz]{}



 \email{mourad.choulli@univ-lorraine.fr}
 \email{laurent.kayser@univ-lorraine.fr}
 \email{Elmaati.Ouhabaz@math.u-bordeaux.fr}

\thanks{The research of E.M.O.  was partially supported by the ANR
project HAB, ANR-12-BS01-0013-02}

\begin{document}
\maketitle

\centerline{\scshape Mourad Choulli, Laurent Kayser}
\medskip
{\footnotesize
\centerline{Institut \'Elie Cartan de Lorraine, UMR CNRS 7502, Universit\'e de Lorraine}
   \centerline{Boulevard des Aiguillettes, BP 70239, 54506 Vandoeuvre les Nancy cedex -}
   \centerline{ Ile du Saulcy, 57045 Metz cedex 01, France}
} 

\medskip

\centerline{\scshape El Maati Ouhabaz}
\medskip
{\footnotesize
 \centerline{Institut Math\'ematiques de Bordeaux, UMR CNRS 5251}
   \centerline{Universit\'e de Bordeaux, 351 Cours de la Lib\'eration}
   \centerline{F-33405 Talence, France}
}

\bigskip

\begin{abstract}
Given a domain $\Omega$ of a complete Riemannian manifold $\mathcal{M}$ and define $\mathcal{A}$ to be the Laplacian with Neumann boundary condition on $\Omega$. We prove that, under appropriate conditions, the corresponding heat kernel satisfies the  Gaussian upper bound
 \[
h(t,x,y)\leq \frac{C}{\left[V_\Omega(x,\sqrt{t})V_\Omega (y,\sqrt{t})\right]^{1/2}}\left( 1+\frac{d^2(x,y)}{4t}\right)^{\delta}e^{-\frac{d^2(x,y)}{4t}},\;\; t>0,\; x,y\in \Omega .
\]
Here $d$ is the geodesic distance on $\mathcal{M}$, $V_\Omega (x,r)$ is the Riemannian volume of $B(x,r)\cap \Omega$, where $B(x,r)$ is the geodesic ball of center $x$ and radius $r$, and $\delta$ is a constant related to the doubling property of $\Omega$.

\smallskip
As a consequence we obtain  analyticity of the semigroup $e^{-t {\mathcal A}}$ on $L^p(\Omega)$ for all $p \in [1, \infty)$ as well as a spectral multiplier result.
\end{abstract}

%=========================================================

\section{Introduction and main results}\label{section1}

This short note is devoted to the Gaussian upper bound for  the heat kernel of the Neumann Laplacian.  Let us start with the Euclidean setting in which $\Omega$ is a bounded Lipschitz domain of $\mathbb{R}^n$. Let $\Delta_N$ be the  Neumann Laplacian. It is well known  that the corresponding heat kernel  $h(t,x,y)$ satisfies
\begin{equation}\label{gau1}
0 \le h(t,x,y) \le C t^{-n/2} e^{t} e^{-c\frac{ |x-y|^2}{t}} ,\;\;  t > 0,\;  x, y \in \Omega.
\end{equation}
 One can replace the extra term $e^t$ by $(1 + t)^{n/2}$ but the decay $h(t,x,y) \le Ct^{-n/2}$ cannot hold for large $t$ 
 since $e^{t \Delta_N} 1 = 1$. We refer to the monographs \cite{Da} or \cite{Ou} for more details.
 
 \smallskip
 In  applications, for example when applying the Gaussian bound to obtain  spectral multiplier results one can apply  \eqref{gau1}
 to $-\Delta_N + I$ (or $\epsilon I$ for any $\epsilon > 0$) and not to $-\Delta_N$. It is annoying to add the identity operator especially it is not clear how the functional calculus for $-\Delta_N$ can be related to that of $-\Delta_N + I$. The same problem  occurs for analyticity of the semigroup $e^{t\Delta_N}$ on $L^p(\Omega)$ for $p \in [1, \infty)$. One obtains from \eqref{gau1} analyticity of the semigroup but not a bounded analytic semigroup.  This   boundedness (on sectors  of the right half plane)  is important in order to obtain appropriate estimates for the resolvent  or for the time derivatives of the  solution to the corresponding evolution equation on $L^p$. 
 In this note we will show in an elementary way how one can resolve this question. The 
 idea is  that \eqref{gau1} can be improved into a Gaussian upper bound of the type
 \begin{equation}\label{gau2}
h(t,x,y)\leq \frac{C}{\left[V_\Omega(x,\sqrt{t})V_\Omega (y,\sqrt{t})\right]^{1/2}}e^{-c \frac{|x-y|^2}{t}},\;\;  t>0,\; x,y\in \Omega,
\end{equation}
where $V_\Omega(x,r)$ denotes the volume of $\Omega \cap B(x, r)$ and  $B(x,r)$ is the open ball of center $x$ and radius $r$. There is no extra factor in \eqref{gau2} and one can use this estimate in various applications of Gaussian bounds  instead of \eqref{gau1}. 

\smallskip
We shall state most of the results for Lipschitz domains of general Riemannian manifolds. 

\smallskip
Let $(\mathcal{M}, g)$ be a complete Riemannian manifold of dimension $n$ without boundary. Let $\Omega$ be a subdomain of $\mathcal{M}$ with Lipschitz boundary $\Gamma$. That is, $\Gamma$ can be described in an appropriate local coordinates by means of graphs of Lipschitz functions. Specifically, for any $p\in \Gamma$, there exist a local chart $(U,\psi )$, $\psi:U\rightarrow \mathbb{R}^n$ with $\psi(p)=0$, a Lipschitz function $\lambda:\mathbb{R}^{n-1}\rightarrow \mathbb{R}$ with $\lambda (0)=0$ and $\epsilon >0$ such that
\begin{align*}
&\psi (U\cap \Omega )=\{(x',\lambda (x')+t);\; 0<t<\epsilon ,\; x'\in \mathbb{R}^{n-1},\; |x'|<\epsilon\},
\\
&\psi (U\cap \Gamma )=\{(x',\lambda (x'));\;  x'\in \mathbb{R}^{n-1},\; |x'|<\epsilon\}.
\end{align*}

We use  in this text Einstein summation convention for repeated indices. We recall that, in local coordinates $x=(x_1,\ldots ,x_n)$,
\[
g(x)=g_{ij}dx_i\otimes dx_j.
\]
If $f\in C^\infty (\mathcal{M})$,  the gradient of $f$ is the vector field given by
\[
\nabla f= g^{ij}\frac{\partial f}{\partial x_i}\frac{\partial}{\partial x_j}
\]
and the Laplace-Beltrami operator is the operator acting as follows
\[
\Delta f=|g|^{-1/2}\frac{\partial}{\partial x_i}\left( |g|^{1/2}g^{ij}\frac{\partial f}{\partial x_j}\right),
\]
where $(g^{ij})$ is the inverse of the metric $g$ and $|g|$ is the determinant of $g$.

\smallskip
Let $\mu$ be the Riemannian measure induced by the metric $g$. That is
\[
d\mu = |g|^{1/2}dx_1\ldots dx_n.
\]

We set $L^2(\Omega )=L^2(\Omega , d\mu)$. Let $H^1(\Omega )$ be the closure of $C_0^\infty (\overline{\Omega})$ with respect to the norm
\[
\|f\|_{H^1(\Omega )}=\left(\int_\Omega f(x)^2d\mu (x)+\int_\Omega |\nabla f(x)|^2d\mu (x)\right)^{1/2}.
\]
Here
\[
|\nabla f|^2=\langle \nabla f,\nabla f\rangle,
\]
where
\[
\langle \nabla f,\nabla g\rangle = g^{ij}\frac{\partial f}{\partial x_i}\frac{\partial g}{\partial x_j}.
\]

We consider on $L^2(\Omega )\times L^2(\Omega )$ the unbounded bilinear form 
\[
\mathfrak{a}(f,g)=\int_\Omega \langle \nabla f,\nabla g\rangle d\mu(x)
\]
with domain $D(\mathfrak{a})=H^1(\Omega )$.
 
\smallskip
Since $\Gamma$ is Lipschitz, the unit conormal $\nu \in T^\ast \mathcal{M}$ is defined a.e. with respect to the surface measure $d\sigma$. Let $\partial _\nu f=\langle \nabla f,\nu \rangle =g^{ij}\nu _i\frac{\partial f}{\partial x_j}$ and
\[
H_\Delta (\Omega )=\{f\in L^2(\Omega );\; \Delta f\in L^2(\Omega )\}.
\]
We recall the Green's formula
\[
\int_\Omega \langle \nabla f, \nabla g\rangle d\mu =-\int_\Omega \Delta fgd\mu +\int_\Gamma \partial _\nu fg d\sigma ,\;\; f\in C_0^\infty (\overline{\Omega}),\; g\in H^1(\Omega ).
\]
In light of this formula, we  define $\partial _\nu f$, $f\in H_\Delta (\Omega )$, as an element of $H^{-1/2}(\Gamma )$, the dual space of $H^{1/2}(\Gamma )$, by the following formula
\[
(\partial _\nu f,g)_{1/2}:=\int_\Omega \Delta fgd\mu + \int_\Omega \langle \nabla f, \nabla g\rangle d\mu ,\;\;  g\in H^1(\Omega ).
\]
Here $(\cdot ,\cdot )_{1/2}$ is the duality pairing between $H^{1/2}(\Gamma )$ and $H^{-1/2}(\Gamma )$.

\smallskip
We define the operator $\mathcal{A}u=-\Delta u$ with domain 
\[
D(\mathcal{A})=\{ u\in H_\Delta (\Omega );\; \partial _\nu u=0\}.
\]
Then it is straightforward to see that $\mathcal{A}$ is the operator associated to the form $\mathfrak{a}$.

\smallskip
Let $d$ be the geodesic distance and $B(x,r)$ be the geodesic ball with respect to $d$ of center $x\in \mathcal{M}$ and radius $r>0$, and set  $V(x,r)=\mu (B(x,r))$. 

\smallskip
We assume in what follows that $\mathcal{M}$ satisfies the volume doubling (abbreviated to $VD$ in the sequel) property: there exists $C>0$ so that
\[
V(x,2r)\leq CV(x,r),\;\; x\in \mathcal{M},\; r>0.
\]
We shall assume that the heat kernel $p(t,x,y)$ of the Laplacian on $\mathcal{M}$ satisfies the Gaussian upper bound
\begin{equation}\label{gauM}
p(t,x,y) \leq \frac{C}{\left[V(x,\sqrt{t})V(y,\sqrt{t})\right]^{1/2}} e^{-c\frac{d^2(x,y)}{t}},\;\; t>0,\; x,y\in \mathcal{M}
\end{equation}
in which $C$ and $c$ are positive constants. 

\smallskip
A typical example of a  manifold  which satisfies both properties is a manifold with non negative Ricci curvature. The volume doubling property is then an immediate consequence of Gromov-Bishop theorem. The Gaussian upper bound can be found in  \cite{LY}.

\smallskip
We define $V_\Omega$ by
\[
V_\Omega(x,r)=\mu \left( B(x,r)\cap \Omega \right),\;\; r>0,\; x\in \Omega .
\]
The main assumption on $\Omega$ is the following  variant of  the $VD$ property: there exist two constants $K>0$ and $\delta >0$ so that
\begin{equation}\label{vdi}
V_\Omega(x,s)\leq K\left( \frac{s}{r}\right)^\delta V_\Omega (x,r),\;\; 0<r\leq s ,\; x\in \Omega .
\end{equation}

Note that this doubling property holds for all bounded Lipschitz domains of $\mathbb{R}^n$ (with $\delta = n$). We shall discuss this in Section \ref{sec3}.

\smallskip
Most of the results we will refer to are valid for metric measure space with Borel measure. In our case this metric measure space is nothing else but $(\Omega ,d,\mu )$. Here, we keep the notations $d$ and $\mu$ for the distance and measure induced on $\Omega$ by $d$ on $\mathcal{M}$ and $\mu$ on $\mathcal{M}$.

\smallskip
Now we state our main results  which we formulate in  following theorem and the subsequent corollaries. 

\begin{theorem}\label{pthm}
(1) $-\mathcal{A}$ generates a symmetric Markov semigroup $e^{-t\mathcal{A}}$ with kernel $h\in C^\infty ((0,\infty )\times \Omega \times \Omega )$.
\\
(2) Suppose that $\mathcal{M}$  satisfies $VD$ and \eqref{gauM} and $\Omega$ satisfies the $VD$ property \eqref{vdi} and $\textrm{diam}\, (\Omega )<\infty$. Then  $h$ has the following Gaussian upper bound
\[
h(t,x,y)\leq \frac{C}{\left[V_\Omega(x,\sqrt{t})V_\Omega (y,\sqrt{t})\right]^{1/2}}\left( 1+\frac{d^2(x,y)}{4t}\right)^{\delta}e^{-\frac{d^2(x,y)}{4t}},\;\; t>0,\; x,y\in \Omega .
\]
\end{theorem}

Since $\rho \in [0,\infty )\rightarrow (1+\rho )^\delta e^{-\rho /2}$ is bounded function, an immediate consequence of Theorem \ref{pthm} is

\begin{corollary}\label{gub}
Suppose that $\mathcal{M}$ satisfies $VD$ and \eqref{gauM} and $\Omega$ satisfies the $VD$ property \eqref{vdi} and $\textrm{diam}\, (\Omega )<\infty$.  Then 
\[
h(t,x,y)\leq \frac{C}{\left[V_\Omega (x,\sqrt{t})V_\Omega (y,\sqrt{t})\right]^{1/2}}e^{-\frac{d^2(x,y)}{8t}},\;\; t>0,\; x,y\in \Omega .
\]
\end{corollary}

We note that for unbounded domains, Gaussian upper bounds for the Neumann heat kernel are proved in \cite{GS}.

\smallskip
Theorem \ref{pthm} (2) or its corollary has several consequences.

\begin{corollary}\label{ana}
Suppose that $\mathcal{M}$ satisfies $VD$ and \eqref{gauM} and $\Omega$ satisfies the $VD$ property \eqref{vdi} and $\textrm{diam}\, (\Omega )<\infty$. Then\\
(1) the semigroup $e^{-t\mathcal{A}}$ extends to a bounded holomorphic semigroup on $\mathbb{C}^+$ 
on $L^p(\Omega ,\mu )$ for all $p \in [1, \infty)$,
\\
(2) the spectrum of $\mathcal{A}$, viewed as an operator acting on $L^p(\Omega )$, $p\in [1, \infty )$, is independent of $p$.
\end{corollary}

Assertion (1) is a consequence of Corollary \ref{gub} combined with \cite[Corollary 7.5, page 202]{Ou}. It was originally proved in \cite{Ou95}. Assertion (2)  follows from a result in \cite{Dav} which asserts that a Gaussian upper bound implies $p$-independence  of the spectrum. See also  \cite[Theorem 7.10, page 206]{Ou} for the general form needed here. 

\smallskip
Let $(E_\lambda )$ be the spectral resolution of the non negative self-adjoint operator $\mathcal{A}$. We recall that for any bounded Borel function $f: [0,\infty )\rightarrow \mathbb{C}$, the operator $f(\mathcal{A})$ is defined by
\[
f(\mathcal{A})=\int _0^\infty f(\lambda )dE_\lambda .
\]
An operator $T$ on the measure space $(\Omega ,\mu)$ is said of weak type $(1,1)$ if
\[
\|T\|_{L^1(\Omega )\rightarrow L^1_w(\Omega )}:=\sup \{\lambda \mu (\{x\in \Omega ;\; |T\varphi (x)|>\lambda \});\; \lambda >0,\; \|\varphi \|_{L^1(\Omega )}=1\}<\infty .
\]

In light of \cite[Theorem 1.3, page 450 and Remark 1, page 451]{DOS}, another consequence of Corollary \ref{gub} is

\begin{corollary}\label{hc}
Suppose that $\mathcal{M}$ satisfies $VD$ and \eqref{gauM} and $\Omega$ satisfies the $VD$ property \eqref{vdi} and $\textrm{diam}\, (\Omega )<\infty$.  Let $s>\delta /2$, where $\delta$ is as in \eqref{vdi}, $\varphi \in C_0^\infty ((0,\infty))$ non identically equal to zero and $f:[0,\infty )\rightarrow \mathbb{C}$ a Borel function satisfying
\[
\sup_{t>0} \| \varphi (\cdot )f(t\cdot )\|_{W^{s,\infty}}<\infty .
\]
Then $f(\mathcal{A})$ is of weak type $(1,1)$  and bounded on $L^p(\Omega )$ for any $p\in (1,\infty )$. Additionally,
\[
\|f(\mathcal{A})\|_{L^1(\Omega )\rightarrow L^1_w(\Omega )}\leq C_s\left( \sup_{t>0} \| \varphi (\cdot )f(t\cdot )\|_{W^{s,\infty}}+|f(0)|\right).
\]
\end{corollary} 

A particular case of this corollary concerns the imaginary powers of $\mathcal{A}$. Precisely, $\mathcal{A}^{ir}$, $r\in \mathbb{R}$, extends to a bounded operator on $L^p(\Omega )$, $p\in (1,\infty )$, and, for any $\epsilon >0$, there is a constant $C_\epsilon >0$ so that
\begin{equation}\label{imp}
\|\mathcal{A}^{ir}\|_{\mathscr{B}(L^p(\Omega ))}\leq C_\epsilon (1+|r|)^{\delta |1/2-1/p|+\epsilon}.
\end{equation}
Indeed, an application of the previous corollary with $f(\lambda) = \lambda^{ir}$ shows that 
\[
\|\mathcal{A}^{ir} \|_{L^1(\Omega )\rightarrow L^1_w(\Omega )}\leq C_\epsilon (1+|r|)^{\delta/2 +\epsilon}.
\]
On the other hand, the standard functional calculus for  self-adjoint operators gives
 \[
\|\mathcal{A}^{ir}\|_{\mathscr{B}(L^2(\Omega ))}\leq 1.
\]
Therefore, \eqref{imp} follows by interpolation. We refer to \cite[Corollary 7.24, page 239]{Ou} for more details. 
%=============================

\section{Proof of the main theorem}\label{section2}

\begin{proof}[Proof of Theorem \ref{pthm}] (1) We first recall that $-\mathcal{A}$ generates on $L^2(\Omega )$ an analytic semigroup $e^{-t\mathcal{A}}$. Note that  
\[
e^{-t\mathcal{A}}=\int_0^{+\infty} e^{-t\lambda} dE_\lambda ,\;\; t\geq 0.
\] 

\begin{proposition}\label{proposition1}
(a) $e^{-t\mathcal{A}}$ is positivity preserving.
\\
(b) $e^{-t\mathcal{A}}$ is a contraction on $L^p(\Omega )=L^p(\Omega ,d\mu )$ for all $1\leq p\leq \infty$ and $t \geq 0$.
\end{proposition}

\begin{proof}
(a) We recall that if $u\in H^1(\Omega )$, then $u^+,u^-\in H^1(\Omega )$ and $\nabla |u|=\nabla u^++\nabla u^-$. Hence
\[
\mathfrak{a}(|u|,|u|)=\mathfrak{a}(u,u),\;\; u\in H^1(\Omega ).
\]
In light of \cite[Theorem 1.3.2, page 12]{Da}, we deduce that $e^{-t\mathcal{A}}$ is positivity preserving.

\medskip
(b) If $0\leq u\in H^1(\Omega )$, then one can check in a straightforward manner that $u\wedge 1=\min (u,1)\in H^1(\Omega )$ and
\[
\nabla (u\wedge 1)=\left\{ \begin{array}{ll} \nabla u\; &{\rm in}\;  [u>1],\\ 0 &{\rm in}\;  [u\leq 1]. \end{array}\right.
\]
Therefore $e^{-t\mathcal{A}}$ is a contraction semigroup on $L^p(\Omega )$ for all  $1\leq p\leq \infty$ 
 by \cite[Theorem 1.3.3, page 14]{Da}.
\end{proof}

This proposition says that $e^{-t\mathcal{A}}$ is a symmetric Markov semigroup.

\smallskip
We have for any integer $k$,
\begin{equation}\label{hk1}
\mathcal{A}^ke^{-t\mathcal{A}}= \int_0^{+\infty} \lambda ^ke^{-t\lambda} dE_\lambda .
\end{equation}
Therefore, $e^{-t\mathcal{A}}f\in D(\mathcal{A})$, for all  $f\in L^2(\Omega )$ and $t>0$.

\smallskip
On the other hand, we get from the usual interior elliptic regularity
\[
\bigcap_{k\in \mathbb{N}}D(\mathcal{A}^k)\subset C^\infty (\Omega ).
\]
Hence, $x\rightarrow e^{-t\mathcal{A}}f(x)$ belongs to $C^\infty (\Omega )$ for any fixed $t>0$. But, $t\rightarrow e^{-t\mathcal{A}}f$ is analytic on $(0,\infty )$ with values in the Hilbert space $D(\mathcal{A}^k)$. Consequently, $(t,x)\rightarrow e^{-t\mathcal{A}}f(x)$ is in $C^\infty ((0,\infty )\times \Omega )$.

\smallskip
From now on, the scalar product of $L^2(\Omega )$ will be denoted by $(\cdot ,\cdot )_{2,\Omega}$ and the norm of $L^p(\Omega )$, $1\leq p\leq \infty$, by $\|\cdot \|_{p,\Omega}$. The norm of $L^p(\mathcal{M})$ is simply denoted by $\|\cdot \|_p$, $1\leq p\leq \infty$.

\smallskip
We fix $t>0$. Using that $\lambda \rightarrow \lambda ^ke^{-t\lambda}$ attains its maximum value at $\lambda =k/t$, we obtain  from \eqref{hk1} for $f\in L^2(\Omega )$
\begin{align}
\|\mathcal{A}^ke^{-t\mathcal{A}}f\|_{2,\Omega}^2&=\int_0^\infty [\lambda ^ke^{-\lambda t}]^2d\|E_\lambda f\|_{2,\Omega}^2\label{hk2}
\\
&\leq \sup_{\lambda >0}[\lambda ^ke^{-\lambda t}]^2\int_0^\infty d\|E_\lambda f\|_{2,\Omega}^2\nonumber
\\
&\leq \left(\frac{k}{t}\right)^{2k}e^{-2k}\|f\|_{2,\Omega}^2. \nonumber
\end{align}
Again by the interior elliptic regularity, $D(\mathcal{A}^k)$ is continuously embedded in $C(\Omega )$ when $k$ is sufficiently large. This and \eqref{hk2} entails: for any $\omega \Subset \Omega$, there exists $C=C(\Omega ,\omega ,k)$ so that
\begin{equation}\label{hk3}
\sup_{\overline{\omega}}|e^{-t\mathcal{A}}f|\leq  \frac{C}{t^k}^2\|f\|_{2,\Omega}.
\end{equation}
In particular, for any fixed $x\in \Omega$ and $t>0$, the (linear) mapping $f\rightarrow e^{-t\mathcal{A}}f(x)$ is continuous. We can then apply the Riesz representation theorem to deduce that there exists $\ell (t,x)\in L^2(\Omega )$ so that
\[
e^{-t\mathcal{A}}f(x)=(\ell (t,x),f)_{2,\Omega},\;\; x\in \Omega ,\; t>0.
\]
Therefore, $(t,x)\rightarrow \ell (t,x)\in L^2(\Omega )$ is weakly $C^\infty$ on $(0,\infty )\times \Omega$ and hence norm $C^\infty$ by \cite[Section 1.5]{da}. 

\smallskip
Let $h(t,x,y)=(\ell (t/2,x),\ell (t/2,y))$. Then $h\in C^\infty ((0,\infty )\times \Omega\times \Omega )$ and 
\[
\left(e^{-t\mathcal{A}}f,g\right)_{2,\Omega}=\left(e^{-\frac{t}{2}\mathcal{A}}f,e^{-\frac{t}{2}\mathcal{A}}g\right)_{2,\Omega}=\int_\Omega \int_\Omega h(t,x,y)f(x)g(y)d\mu(x)d\mu (y),\;\; f,g\in C_0^\infty (\Omega ).
\]
By the density of $C_0^\infty (\Omega )$ in $L^2(\Omega )$, we derive from the last identity that
\[
e^{-t\mathcal{A}}f(x)=\int_\Omega h(t,x,y)f(x)d\mu (x),\;\; t>0,\; x\in \Omega ,\; f\in L^2(\Omega ).
\]

(2) We start with the following proposition. 

\begin{proposition}\label{proposition2}
$e^{-t\mathcal{A}}$ satisfies the Davies-Caffney (abbreviated to DG in the sequel) property. That is, for any $t>0$, $U_1$, $U_2$ open subsets of $\Omega$, $f\in L^2(U_1,d\mu )$ and $g\in L^2(U_2,d\mu )$,
\[
\left| (e^{-t\mathcal{A}}f,g)_{2,\Omega}\right|\leq e^{-\frac{r^2}{4t}}\|f\|_{2,\Omega }\|g\|_{2,\Omega}.
\]
Here
\[
r={\rm dist}(U_1,U_2)=\inf_{x\in U_1,\; y\in U_2}d(x,y).
\]
\end{proposition}

\begin{proof}
We omit  the proof which is similar to that of \cite[Theorem 3.3, page 515]{CS}.
\end{proof} 

We now observe that $\Omega$ has the $1$-extension property (see for instance \cite[Theorem C]{MMS}). In other words, there exists $\mathcal{E}\in \mathscr{B}(H^1(\Omega ),H^1(\mathcal{M}))$ satisfying $(\mathcal{E}u)_{|\Omega}=u$, $u\in H^1(\Omega )$.

\smallskip
On the other hand, since $\mathcal{M}$ has the volume doubling property and the Gaussian bound \eqref{gauM}, it follows from 
\cite[Theorem 1.2.1]{BCS}  that the following Gagliado-Nirenberg type inequality holds: for $2<q\leq +\infty$, there exists a constant $C>0$ so that
\begin{equation}\label{GN}
\|fV^{\frac{1}{2}-\frac{1}{q}}(\cdot ,r)\|_q\leq C\left (\|f\|_2+r\||\nabla f|\|_2^2\right),\;\; r>0,\; f\in C_0^\infty (\mathcal{M}).
\end{equation}

In light of \eqref{GN} and using that $V_\Omega (\cdot ,r)\leq V(\cdot ,r)$ in $\Omega$, we obtain for
 $r>0$, $f\in H^1(\Omega )$ and fixed $2<q\leq \infty$, 
\begin{align*}
\|fV_\Omega ^{\frac{1}{2}-\frac{1}{q}}(\cdot ,r)\|_{q,\Omega } &\leq \|fV^{\frac{1}{2}-\frac{1}{q}}(\cdot ,r)\|_{q,\Omega }
\\
&\leq \|(\mathcal{E}f)V^{1/2-1/q}\|_q
\\
&\leq C\left( \|\mathcal{E}f\|_2+r\| |\nabla (\mathcal{E}f)|\|_2\right)
\\
& \leq C\|\mathcal{E}\|  \left( (1+r)\|f\|_{2,\Omega }+r\| |\nabla f|\|_{2,\Omega }\right) .
\end{align*}
Here $\|\mathcal{E}\|$ is the norm of $\mathcal{E}$ in $\mathscr{B}(H^1(\Omega ),H^1(\mathcal{M}))$. Hence
\begin{equation}\label{gn}
\|fV_\Omega (\cdot ,r)^{\frac{1}{2}-\frac{1}{q}}\|_{q,\Omega }\leq C \left( \|f\|_{2,\Omega }+r\| |\nabla f|\|_{2,\Omega }\right),\;\; r>0,\; f\in H^1(\Omega ),
\end{equation}
where we used the fact that $V_\Omega (\cdot ,r)=V_\Omega (\cdot ,r_0)=\mu (\Omega )$, for all $r\geq r_0=\textrm{diam}\, (\Omega )$.

\smallskip
We then apply \cite[Theorem 1.2.1]{BCS} to derive that $h$ possesses a diagonal upper bound. In other words, there exists a constant $C>0$ so that
\begin{equation}\label{dub}
h(t,x,y)\leq \frac{C}{\left[V_\Omega (x,\sqrt{t})V_\Omega (x,\sqrt{t})\right]^{1/2}},\;\; t>0,\; x, y\in \Omega .
\end{equation}
Since $e^{-t\mathcal{A}}$ has the $DG$ property by Proposition \ref{proposition2} we get, from \cite[Corollary 5.4, page 524]{CS},
\[
h(t,x,y)\leq \frac{eC}{\left[V_\Omega (x,\sqrt{t})V_\Omega (y,\sqrt{t})\right]^{1/2}}\left( 1+\frac{d^2(x,y)}{4t}\right)^{\delta}e^{-\frac{d^2(x,y)}{4t}},\;\; t>0,\; x,y\in \Omega .
\]
The proof is then complete.
\end{proof}

%============================================

\section{Domains with volume doubling property}\label{sec3}

\subsection*{Flat case}  It is known that any bounded Lipschitz domain of $\mathbb{R}^n$ satisfies the volume doubling property. 
We discuss this again here. We consider $\mathbb{R}^n$ equipped with its euclidean metric $g=(\delta_{ij})$. Let
\[
\mathscr{C}(y,\xi ,\epsilon )=\{ z\in \mathbb{R}^n;\; (z-y)\cdot \xi \geq (\cos \epsilon )|z-y|,\; 0<|y-z|<\epsilon \},
\]
where $y\in\mathbb{R}^n$, $\xi \in \mathbb{S}^{n-1}$ and $0<\epsilon $. That is, $C(y,\xi ,\epsilon)$  is the cone, of dimension $\epsilon$,  with vertex $y$, aperture $\epsilon$ and directed by $\xi$.

\smallskip
We say that $\Omega$ has the $\epsilon$-cone property if
\[
\textrm{for any}\; x\in \Gamma ,\; \textrm{there exists}\; \xi _x\in \mathbb{S}^{n-1}\; \textrm{so that, for all}\; y\in \overline{\Omega}\cap B(x,\epsilon),\; \mathscr{C}(y,\xi _x,\epsilon )\subset \Omega .
\]

Let $\Omega $ be a bounded Lipschitz domain of $\mathbb{R}^n$. Then, by \cite[Theorem 2.4.7, page 53]{HP}, $\Omega$ has the $\epsilon$-cone property, for some $\epsilon >0$. This implies that there exist $c_0>0$ and $\rho >0$ so that
\begin{equation}\label{Ex1}
V_\Omega (x,r)=|B(x,r)\cap \Omega |\ge c_0r^n,\;\; x\in \Omega ,\; 0<r\le \rho .
\end{equation}

An immediate  consequence is that $\Omega$ (equipped with its euclidean metric) satisfies the volume doubling property.  
Indeed, let $r_0=\textrm{diam}\, (\Omega )$ and $0<r\leq s$. Then \eqref{Ex1} entails
\begin{equation}\label{Ex2}
V_\Omega (x,s)\leq c_1s^n=c_1\left( \frac{s}{r}\right)^nr^n\leq \frac{c_1}{c_0}\left( \frac{s}{r}\right)^nV_\Omega (x,r),\;\; 0<r\le \rho ,
\end{equation}
where $c_1=|B(0,1)|$.

\smallskip
Also, when $\rho <r_0$,
\begin{equation}\label{Ex3}
V_\Omega (x,s)\leq \frac{c_1}{c_0}\left( \frac{s}{\rho }\right)^nV_\Omega (x,\rho )\le \frac{c_1}{c_0}\left(\frac{r_0}{\rho}\right)^n\left( \frac{s}{r}\right)^nV_\Omega (x,r ),\;\; \rho <r\le r_0.
\end{equation}
Finally, it is obvious that
\begin{equation}\label{Ex4}
V_\Omega (x,s)=|\Omega |=V_\Omega (x,r_0)\leq \left(\frac{s}{r}\right)^n V_\Omega (x,r ),\;\; r>r_0.
\end{equation}
Estimates \eqref{Ex2}, \eqref{Ex3} and \eqref{Ex4} show the volume doubling property.

\subsection*{Manifold with sectional curvature bounded from above}

Let $T_x\mathcal{M}$ be the tangent space at $x\in \mathcal{M}$, $\mathbb{S}_x\subset T_x\mathcal{M}$ the unit tangent sphere and $S\mathcal{M}$ the unit tangent bundle. Let $\Phi_t$ be the geodesic flow with phase space $S\mathcal{M}$. That is, for any $t\geq 0$, 
\[
\Phi_t:S\mathcal{M}\rightarrow S\mathcal{M}: (x,\xi )\in S\mathcal{M} \rightarrow\Phi_t (x,\xi )=\left( \gamma_{x,\xi }(t),\dot{\gamma}_{x,\xi }(t)\right).
\]
Here $\gamma_{x,\xi}:[0,\infty )\rightarrow \mathcal{M}$ is the unit speed geodesic starting at $x$ with tangent unit vector $\xi$ and $\dot{\gamma}_{x,\xi}(t)$ is the unit tangent vector to $\gamma_{x,\xi}$ at $\gamma_{x,\xi}(t)$ in the forward $t$ direction.

\smallskip
If $(x,\xi)\in S\mathcal{M}$, we denote by $r(x,\xi)$ the distance from $x$ to the cutlocus in the direction of $\xi$:
\[
r(x,\xi )=\inf\{t>0;\; d(x,\Phi_t(x,\xi ))<t\}.
\]

We fix $\delta \in (0,1]$ and $r>0$.  Following \cite{Sa}, a $(\delta ,r)$-cone at $x\in \mathcal{M}$ is the set of the form
\[
\mathscr{C}(x,\omega _x,r)=\{ y=\gamma_{x,\xi}(s);\; \xi\in \omega _x,\; 0\leq s<r\},
\]
where $\omega _x$ is a subset of $\mathbb{S}_x$ so that $r<r(x,\xi )$ for all $\xi \in \omega_x$ and $|\omega _x |\geq \delta$ (here $|\omega _x|$ is the volume of $\omega _x$ with respect to the normalized measure on the sphere $\mathbb{S}_x$).

\smallskip
A domain $D$ which contains a $(\delta ,r)$-cone at $x$ for any $x \in D$ is said to satisfy
the interior $(\delta ,r)$-cone condition.

%\smallskip
%We observe that if $\mathcal{C}$ is a closed strongly convex subset of $\mathcal{M}$, then $\Omega =\mathcal{M}\setminus \mathcal{C}$ has the $(1/2 ,r)$-cone %condition, for some $r$ (this fact follows from the same argument to that in \cite[Example 8.1, page 370]{Sa}).

\smallskip
%We recall that a complete simply connected Riemannian manifold without boundary with nonpositive sectional curvature is usually called  a Cartan-Hadamard %manifold. 
Let
\[
s_\kappa (r)=\left\{
\begin{array}{ll}
\left(\frac{\sin (\sqrt{\kappa}r}{\sqrt{\kappa}}\right)^{n-1}\;\; &\textrm{if}\; \kappa >0,
\\
r^{n-1} &\textrm{if}\; \kappa =0,
\\
\left(\frac{\sinh (\sqrt{-\kappa}r}{\sqrt{-\kappa}}\right)^{n-1}\  &\textrm{if}\; \kappa <0.
\end{array}
\right.
\]

\smallskip
We  assume that the sectional curvature of $\mathcal{M}$ is bounded above by a constant $\kappa$, $\kappa \in \mathbb{R}$,  and $\Omega$  satisfies the interior $(\delta ,r)$-cone condition. Let $J(x,\xi ,t)$ be the density of the volume element in geodesic coordinates around $x$. That is
\[
dV(y)=J(x,\xi ,t)d_{\mathbb{S}_x}dt,\;\; y=\gamma _{x,\xi}(t),\; t<r(x,\xi ).
\]
By an extension of G\"unther's comparison theorem (see for instance \cite{KK}),  $J$ satisfies the following uniform lower bound
%\begin{equation}\label{lb}
\[
J(x,\xi ,t)\geq s_\kappa (t).
\]
%\end{equation}
Consequently, for some $r_0>0$,
\begin{equation}\label{ce}
V_\Omega (x,r)\geq V (\mathscr{C}(x,\omega _x,r))\geq c_0 r^n,\;\; x\in \Omega ,\; 0<r\leq r_0 ,
\end{equation}

We proceed similarly to the flat case  to prove the following lemma.

\begin{lemma}\label{lemEx2}
Assume that $\mathcal{M}$ has sectional curvature bounded from above  and satisfies 
following volume growth condition 
\[
V(x,r)\leq c_1r^n, \;\; 0<r\leq r_1,
\]
for some constants $c_1$ and $r_1$. If $\Omega$ is  of finite diameter and satisfies the $(\delta ,r)$-cone condition, then $V_\Omega$ is doubling.
\end{lemma}

%==========================================================================================

\bigskip

\end{document}